\documentclass[11pt,reqno]{amsart}
\usepackage{amscd,amssymb,amsmath,amsthm}
\usepackage{mathtools}
\usepackage[english]{babel}
\usepackage[T1]{fontenc}
\usepackage[arrow,matrix]{xy}
\usepackage{graphicx}
\usepackage{tikz}
\usetikzlibrary{positioning,matrix,arrows,calc}
\usepackage{cite}
\usepackage{dsfont}
\newtheorem{theorem}[subsection]{Theorem}

\newtheorem{proposition}[subsection]{Proposition}

\theoremstyle{definition}
\newtheorem{remark}[subsection]{Remark}
\newtheorem{definition}[subsection]{Definition}

\numberwithin{equation}{section} 

\newcommand{\xn}{{\mathbf{x}}^{(n)}}

\newcommand{\bea}{\begin{eqnarray}}
\newcommand{\eea}{\end{eqnarray}}

\newcommand{\N}{\mathbb N}

\newcommand{\R}{\mathbb{R}}

\def \> {\Rightarrow}
\def \0 {\emptyset}

\def \xo {{\mathbf{x}}^{(0)}}

\DeclareMathOperator{\inte}{int}

\begin{document}
\title[A prey-predator model with three interacting species]{A prey-predator model with three interacting species}

\author{U.U. Jamilov, M. Scheutzow, I. Vorkastner}

\date{\today}

  \address{ U.\ U.\ Jamilov\\ V.I. Romanovskiy Institute of Mathematics,  Uzbekistan Academy of Sciences,
81, Mirzo-Ulugbek str., 100170, Tashkent, Uzbekistan.}
\email{jamilovu@yandex.ru}

\address{M.\ Scheutzow\\ Institut f\"ur Mathematik, MA 7-5, Fakult\"at II,
        Technische Universit\"at Berlin, Stra\ss e des 17.~Juni 136, 10623 Berlin, FRG;}
\email {ms@math.tu-berlin.de}

\address{I.\ Vorkastner\\ Institut f\"ur Mathematik, MA 7-5, Fakult\"at II,
        Technische Universit\"at Berlin, Stra\ss e des 17.~Juni 136, 10623 Berlin, FRG;}
\email {vorkastn@math.tu-berlin.de}

\begin{abstract}

  In this paper we consider a class of discrete time prey-predator models with three interacting species defined on the two-dimensional simplex.
  For some choices of parameters of the operator describing the evolution  of the relative frequencies, we show that the ergodic hypothesis does
  not hold. Moreover, we prove that any order
  Ces\`aro mean of the trajectories diverges. For another class of parameters, we show that all orbits starting from the interior of the simplex
  converge
  to the unique fixed point of the operator while for the remaining choices of parameters all orbits converge to one of the vertices of the simplex. 
  Contrary to many authors we study discrete time models but we include a {\em speed function} $f$ in the dynamics which allows us to approximate the continuous-time
  case arbitrarily well when $f$ is small.  
\end{abstract}

\subjclass[2010] {Primary 37N25, Secondary 92D10.}

\keywords{quadratic stochastic operator, cubic stochastic operator, Volterra cubic operator, non-ergodic operator.}

\maketitle

\section{Introduction}

During the past decades, many mathematical models have been suggested in order to model the time evolution of interacting populations in the most general sense including for example the
evolution of gene frequencies. As an example,  
following \cite{FW1},   consider  the following {\em  Kolmogorov system} of three interacting populations.
   \begin{align}\label{ks}
    \frac{d x_1}{d t}& = x_1f_1(x_1,x_2,x_3), \notag\\
     \frac{d x_2}{d t}& = x_2f_2(x_1,x_2,x_3),  \\
     \frac{d x_3}{d t}& = x_3f_3(x_1,x_2,x_3),\notag\\[2mm]
     x_1(0)&\geq 0, \ \ x_2(0)\geq 0, \ \ x_3(0)\geq 0, \notag
  \end{align}
where $f_i, \, i=1,2,3$ are given functions modelling the growth (or decay) rates of the populations.
Such models arise in biology e.g.~as food chain models \cite{CPPW},\cite{FW2}.
In nature, and in the sea in particular, there are many species or trophic levels where energy, in the form of food, flows from one species to another.
The mass of the total number of individuals in a species is often
referred to as its biomass. 
The ultimate source of energy is the sun, and in the sea, for example, the {\em trophic web} runs through plankton,
fish, sharks up to whales and finally man, with a myriad of species in between. The
species on one trophic level may predate several species below it. In general, models
involve interactions between several species\cite{Mu}.

System \eqref{ks} models a three (or two) trophic level system. Usually, the population described
by $x_1(t)$ will be a prey population which in the absence of competitors or predators will grow
to the carrying capacity of its environment. It is growing on nutrients
outside the system being modeled. $x_3(t)$ will usually be a predator feeding exclusively on populations within the system
($x_1(t)$ or $x_2(t)$ or both) and hence will become extinct if its prey doesn't survive. $x_2(t)$ will be
either a predator or a prey or both. 

One important question in mathematical models like \eqref{ks} or corresponding discrete time models is the long-time behavior of the absolute or relative population sizes.
In particular, it is of great interest to determine whether a particular population becomes extinct in the long run. If this does not happen to a given population in the model,
then one says that the population {\em persists}. We will provide a precise definition in the set-up of our discrete time model in the next section.


\emph{Quadratic stochastic operators} (see e.g. \cite{K1}).
The notion of a quadratic stochastic operator (QSO) was introduced by Bernstein\cite{Br}.
Such operators arise in models of mathematical genetics and model the dynamics of gene frequencies from one generation to the next. QSOs can also
be used as models for the evolution of predator-prey populations. The  theory of QSOs has been developed in many papers (see \cite{GGJ},\cite{K1},\cite{Lyu}, \cite{Ul},\cite{Zakh}), see \cite{GMR} for a recent review.

 Let  $E=\{1,\dots,m\}$ be a finite set. We denote the set of all probability distributions on $E$ by
\begin{equation*}
S^{m-1}=\Big\{ \mathbf{x}=\big(x_1,\dots,x_m\big)\in {\mathbb{R}^m}: x_i\geq 0,\ \,{\forall  i \in E, }\, \sum_{i=1}^m x_i=1 \Big\},
\end{equation*}
which is the $m-1$- dimensional simplex.
 A \emph{quadratic stochastic operator} is a mapping $V:S^{m-1}\mapsto S^{m-1} $ of the simplex to itself of the form $V(\mathbf{x})=\mathbf{x}' \in S^{m-1}$, where
\begin{equation}
\label{qso}
x'_k=\sum_{i,j\in E }p_{ij,k}x_i x_j, \ \ k \in E,
\end{equation}
and the coefficients $p_{ij,k}$  satisfy
\begin{equation}\label{koefqso}
p_{ij,k}=p_{ji,k}\geq 0, \quad  \ \ \sum_{k\in E}p_{ij,k}=1, \ \ i,j,k \in E.
\end{equation}

The \emph{trajectory} (orbit) $\{\mathbf{x}^{(n)}\}, n=0,1,\dots$ of $V$ for an initial {value} $\mathbf{x}^{(0)}\in S^{m-1}$  is defined by
\begin{equation*}
\mathbf{x}^{(n+1)}=V(\mathbf{x}^{(n)})={V^{n+1}(\mathbf{x}^{(0)})}, \quad n=0,1,2,\dots
\end{equation*}

For a nonlinear (quadratic) dynamical system \eqref{qso}, \eqref{koefqso} Ulam \cite{Ul} suggested an analogue of a
measure-theoretic ergodicity in the form of the following ergodic hypothesis: a QSO $V$ is said to be \emph{ergodic} if the limit
\[
\lim_{n\to \infty } \frac{1}{n} \sum_{k=0}^{n-1} V^k(\mathbf{x})
\]
exists for any $\mathbf{x}\in S^{m-1}$.

On the basis of numerical calculations Ulam, in \cite{Ul}, conjectured  that
the ergodic theorem holds for any QSO.

The QSO $V$ is called \emph{Volterra}, if $p_{ij,k}=0$ for any  $k\notin\{i,j\}, \ \ i,j,k \in E$.
In  \cite{Zakh}, Zakharevich  considered the Volterra QSO on $S^2$
\begin{equation}\label{zakhop}
 V:\left\{\begin{array}{l}
   x'_1=x^2_1+2x_1x_2 \\
   x'_2=x^2_2+2x_2x_3 \\
   x'_3=x^2_3+2x_1x_3
 \end{array}\right.
\end{equation}
and showed that it is non-ergodic, that is he proved that Ulam's conjecture is false in general.
Later in \cite{GaZa04}, the authors  established a necessary  condition for a QSO defined on $S^2$  to be a non-ergodic
transformation and thus generalized Zakharevich's result to a larger class of Volterra QSOs.
\cite{GGJ} showed the relation between non-ergodicity of Volterra QSOs and
rock-paper-scissors games which allows to interpret  Zakharevich's example \eqref{zakhop} in terms game theory.
In \cite{ms} the authors generalized the results of papers \cite{Zakh} and \cite{GaZa04}  by considering  a class of Lotka-Volterra operators  defined on
the two-dimensional simplex for which the ergodic theorem  fails.

The biological significance of non-ergodicity of a QSO  is the following: in the long run the behavior of the distribution of species behaves in a chaotic manner. In particular it does not stabilize to an equilibrium nor does it approach a periodic trajectory.\\


In the present paper we consider a class of operators defined on the two-dimensional simplex $S^2$ and study
the asymptotic behavior of the  trajectories generated by this operator. We will see that the asymptotics will heavily depend on the choice of
the parameters. For some choices, the trajectories are non-ergodic, for other choices all orbits starting
from the interior of the simplex converge to the unique fixed point of the operator while
for the remaining choices of parameters all orbits converge to one of the vertices of the
simplex.


\section{{Main results}}
\label{sec:mainresult}

Let $f:S^{2}\mapsto(0,1]$ be continuous and let $a,b,c \in [-1,1]\backslash\{0\}$ be parameters.

We consider the following evolution operator of the population which is a discrete analog of the Kolmogorov model (see \cite{FW1})
of three interacting populations of the form

\begin{equation}\label{cneop}
 W_{f,a,b,c}:\left\{\begin{array}{l}
   x'_1=x_1\Big(1+\big(ax_1x_2-bx^2_3\big)f(\mathbf{x})\Big), \\[2mm]
   x'_2=x_2\Big(1+\big(cx_2x_3-ax^2_1\big)f(\mathbf{x})\Big), \\[2mm]
   x'_3=x_3\Big(1+\big(bx_3 x_1-cx^2_2\big)f(\mathbf{x})\Big).
 \end{array}\right.
\end{equation}

Note that $W_{f,a,b,c}$ maps $S^2$ to $S^2$. The function $f$ is something like a local speed. If  $f(\mathbf{x})$  is close to 0 for all $\mathbf{x}$, then the system moves in small steps and thus resembles a continuous-time system which  we will comment on in Section \ref{conticase}. If the function $f$ is constant, then the operator $W_{f,a,b,c}$ is a {\em cubic stochastic operator} (CSO). For previous studies on CSOs
see  \cite{DJL,HJS,JKL,JL1,RoKh1,RoKh2} and references therein.



We will use the following notation. Let $\inte S^{2}:=\{\mathbf{x}\in S^{2}:  x_1x_2x_3>0\}$ and  $\partial S^{2}:=S^{2}\setminus \inte S^{2}$   be the interior and the boundary of the simplex $S^2$, respectively;\\
a {\em face} of the simplex $S^{2}$ is a set $\Gamma_\alpha=\{\mathbf{x}\in S^{2}: x_i=0, \ \  \, i\notin \alpha\}$, where $\alpha \subset \{1,2,3\}\backslash \emptyset$;\\
let $\mathbf{e}_i=(\delta_{1i},\delta_{2i},\delta_{3i})\in S^{2},\  i=1,2,3$, denote the vertices of the simplex $S^{2}$, where $\delta_{ij}$ is the Kronecker delta.\\
  Let   $\omega(\xo)$ be the set of limit points of the trajectory $\{W_{f,a,b,c}^k(\xo)\in S^{2}: k=0,1,2,\dots \}$.
Observe that $\omega(\xo)$ is non-empty since $S^{2}$ is compact and that $W$ maps $\omega(\xo)$ onto itself. We will sometimes write $W$ instead of
$W_{f,a,b,c}$ and $\xn$ instead of $W^n(\xo)$. Recall the following  definition.

\begin{definition}  A continuous function $\varphi \colon S^{2}\rightarrow \R$ is called a {\em Lyapunov function} (respectively {\em increasing Lyapunov function}) for $W_{f,a,b,c}$ if
$\varphi(W_{f,a,b,c} (\mathbf{x}))\leq \varphi(\mathbf{x})$ (respectively ``$\geq$'')  for all $\mathbf{x} \in S^2$.
\end{definition}

\begin{definition}
 A  point $\mathbf{x}\in S^2$ is called a {\em fixed point} of $W$ if $W(\mathbf{x})=\mathbf{x}$.
\end{definition}

Let us now state two slightly different {\em persistence} properties of $W$ (or the dynamical system generated by $W$).

\begin{definition}\cite{GH} \label{pgh} $W$ is called {\em weakly persistent} if each solution $\mathbf{x}=\mathbf{x}(t)$ of \eqref{cneop} with
$x_1(0)> 0, \ \ x_2(0) > 0, \ \ x_3(0)> 0$ satisfies $ \lim\sup_{t\rightarrow\infty} x_i(t)>0$  for all $i \in \{1,2,3\}$.
\end{definition}

\begin{definition}\cite{FW1} \label{pfw} $W$ is said to be  {\em strongly persistent} if each solution $\mathbf{x}=\mathbf{x}(t)$ of \eqref{cneop} with
$x_1(0)> 0, \ \ x_2(0) > 0, \ \ x_3(0)> 0$ satisfies $ \lim\inf_{t\rightarrow\infty} x_i(t)>0$  for all $i \in \{1,2,3\}$.
\end{definition}

We will now investigate the asymptotics of $W_{f,a,b,c}$ in the case where two of the parameters $a,b,c$ have a different sign, then in case all
parameters are positive and then in case all parameters are negative. It will turn out that in the first case (Theorem \ref{rr}) there is no persistence,
even more: just one species will survive
in the long run. In the second case (Theorem \ref{mr}), weak (but not strong)  persistence holds while in the third case (Theorem \ref{mr2}) the system is strongly persistent.
Note that we excluded the case in which one or more parameters are
zero since these cases are easy and not very interesting.

\begin{theorem}\label{rr}
  If either $ab<0$ or $ac<0$ or $bc<0$ holds, then every trajectory of $W_{f,a,b,c}$ converges to a vertex of $S^2$.
\end{theorem}

\begin{proof} Assume that $ab<0$ (the other cases can be treated analogously).  It is easy to verify that  the faces
$\Gamma_{\{1,2\}}, \Gamma_{\{2,3\}}, \Gamma_{\{1,3\}}$ and $\inte S^2$ are invariant sets with respect to $W_{f,a,b,c}$.
Clearly, the set of fixed  points of $W_{f,a,b,c}$ equals  $\big\{\mathbf{e}_1, \, \mathbf{e}_2, \, \mathbf{e}_3\big\}$.

Assume that $\xo\in \inte S^2$. Then the sequence $x_1^{(n)}$ is strictly increasing if $a>0,\,b<0$ and strictly decreasing if $a<0,\,b>0$,
so the sequence converges. Depending on the sign of $c$, it follows that one of the sequences $x_2^{(n)}$  or $x_3^{(n)}$ is either increasing or decreasing
and therefore converges as well. Since $x_1^{(n)}+ x_2^{(n)}+x_3^{(n)}=1$, the remaining sequence has to converge as well.
Let $x^*:=\lim_{n \to \infty}\mathbf{x}^{(n)}$. Then $x^*$ is a fixed point of $W_{f,a,b,c}$, so $x^*\in \big\{\mathbf{e}_1, \, \mathbf{e}_2, \, \mathbf{e}_3\big\}$.
Note that all trajectories starting in  $\inte S^2$ converge to the same vertex.

It remains to investigate the case $\xo\in \partial S^2$.
%
%
Let $\mathbf{x}^{(0)}\in \Gamma_{1,3}$. Then the restriction of the operator $W_{f,a,b,c}$ to this face has the form
\begin{equation*}
\left\{\begin{array}{l}
   x'_1=x_1(1-bx^2_3f(\mathbf{x})), \\[2mm]
   x'_3=x_3(1+bx_1x_3f(\mathbf{x})).
 \end{array}\right.
\end{equation*}
Therefore, the sequence  $x_1^{(n)}$ is either nondecreasing or nonincreasing and therefore $x^*:=\lim_{n \to \infty}\mathbf{x}^{(n)}$ exists and equals
either $\mathbf{e}_1$ or $\mathbf{e}_3$. The cases $\mathbf{x}^{(0)}\in \Gamma_{1,2}$ and $\mathbf{x}^{(0)}\in \Gamma_{2,3}$ can be treated analogously.
%
%
%
%
%
\end{proof}

The following theorem shows that the long-time behavior of the dynamical system generated by $W_{f,a,b,c}$ is completely different when all
parameters are positive. Let us first introduce Ces\'aro means.

\begin{definition}
  For $\mathbf{x} \in S^2$ and $\mathbf{x}^{(n)}:=W_{f,a,b,c}^n(\mathbf{x})$, $n \in \N_0$ we define the {\em $k$-th order Ces\'aro} sequences by
  $$
  \mathbf{c}^{(n)}_0(\mathbf{x}):=\mathbf{x}^{(n)},\;\mathbf{c}^{(n)}_{k+1}(\mathbf{x}):=    \frac 1{n+1}\sum_{i=0}^n \mathbf{c}^{(i)}_{k}(\mathbf{x}),\; n,k \in \N_0.
$$
\end{definition}

\begin{remark}\label{Cesaro}
  Note that $\mathbf{c}^{(n)}_{k}(\mathbf{x})=\sum_{i=0}^na_{i,k,n}\mathbf{x}^{(i)}$, where $a_{i,k,n}\ge 0$ and $\sum_{i=0}^n a_{i,k,n}=1$ for all $k,n \in \N_0$.
  The coefficients can be recursively computed via
  $$
a_{i,0,n}=\delta_{in},\; a_{i,k+1,n}=\frac 1{n+1}\sum_{j=i}^n a_{i,k,j},\; i,k,n\in \N_0.
  $$
  Further, for each $k \in \N_0$, we have
  \begin{equation}\label{Cesaroproperty}
\lim_{\varepsilon \downarrow 0} \liminf_{n \to \infty}\sum_{i=\lfloor \varepsilon n\rfloor}^n a_{i,k,n}=1.
  \end{equation}
\end{remark}

\begin{theorem}\label{mr}
  If $a>0, b>0, c>0$, then, for each $k \in \N_0$, the vertices  $\mathbf{e}_1$, $\mathbf{e}_2$ and $\mathbf{e}_3$ are limit points of the $k$-th order Ces\'aro sequences
  $\mathbf{c}^{(n)}_{k}(\mathbf{x})$, $n=0,1,2,...$
whenever $\mathbf{x} \in \inte S^2\backslash \{\mathbf{x^*}\}$, where
$$
\mathbf{x^*}:=\bigg({\lambda_1\over \lambda_1+\lambda_2+\lambda_3},{\lambda_2\over \lambda_1+\lambda_2+\lambda_3},{\lambda_3\over \lambda_1+\lambda_2+\lambda_3}\bigg)
$$
and $\lambda_1=\sqrt[3]{bc^2}, \, \lambda_2=\sqrt[3]{ab^2}, \, \lambda_3=\sqrt[3]{ca^2}$.
In particular, $W_{f,a,b,c}$ is a non-ergodic transformation.
\end{theorem}

Before proving the theorem we need some preliminary considerations.
%
It is easy to verify that  the faces $\Gamma_{\{1,2\}}, \Gamma_{\{2,3\}}, \Gamma_{\{1,3\}}$ are invariant sets with respect to $W_{f,a,b,c}$ and that
$\mathbf{e}_1, \, \mathbf{e}_2, \, \mathbf{e}_3,\mathbf{x^*}$ are fixed points. Further,
for $\xo \in \partial S^2$, the trajectory   $W_{f,a,b,c}^n(\xo)$ converges to one of the vertices of $S^2$. In particular,
$W_{f,a,b,c}$ has no fixed points in $\partial S^2\backslash \{\mathbf{e}_1, \, \mathbf{e}_2, \, \mathbf{e}_3\}$.
The following proposition shows, in particular,  that $W_{f,a,b,c}$ has no fixed points in the interior of $S^2$ except $\{\mathbf{x^*}\}$.

\begin{proposition}\label{lf}
  The function $\varphi(\mathbf{x})=x^{\lambda_1}_1x^{\lambda_2}_2x^{\lambda_3}_3$ is a Lyapunov function for  \eqref{cneop} and  $\omega(\xo)$ is an
  infinite subset of $\partial S^2$ 
  for any $\mathbf{x}^{(0)}\in \inte S^2\setminus\{\mathbf{x^*}\}$.
\end{proposition}

\begin{proof}
 Clearly the function $\varphi$ is  continuous on $S^2$ and
\[\varphi(\mathbf{x})=0 \ \ \text{iff} \ \ \mathbf{x} \in \partial S^2, \,
 \max\limits_{\mathbf{x}\in S^2}\varphi(\mathbf{x})=\varphi(\mathbf{x^*}), \text{and } \varphi(\mathbf{x})= \varphi(\mathbf{x^*}) \text{ iff }  \mathbf{x} =\mathbf{x^*}.\]

 From \eqref{cneop} one has  $\varphi(W(\mathbf{x}))=\varphi(\mathbf{x})\psi(\mathbf{x})$, where
\begin{equation}\label{psi}
  \psi(\mathbf{x}) =\Big(1+\big(ax_1x_2-bx^2_3\big)f(\mathbf{x})\Big)^{\lambda_1}
  \Big(1+\big(cx_2x_3-ax^2_1\big)f(\mathbf{x})\Big)^{\lambda_2}\Big(1+\big(bx_3 x_1-cx^2_2\big)f(\mathbf{x})\Big)^{\lambda_3}.
\end{equation}

Using Young's inequality we obtain from \eqref{psi}
\begin{align}
  \psi(\mathbf{x})& \leq\bigg(1+{\big(\lambda_1 ax_1x_2-\lambda_1 bx^2_3+ \lambda_2 cx_2x_3-\lambda_2 ax^2_1+ \lambda_3 bx_3 x_1-\lambda_3 cx^2_2  \big) f(\mathbf{x})\over \lambda_1+\lambda_2+\lambda_3}\bigg)^{\lambda_1+\lambda_2+\lambda_3}\notag\\
  & \leq\bigg(1-{ F(a,b,c,x_1,x_2,x_3)f(\mathbf{x})\over 2(\lambda_1+\lambda_2+\lambda_3)}\bigg)^{\lambda_1+\lambda_2+\lambda_3}\leq1,
\end{align}
where
\[
F(a,b,c,x_1,x_2,x_3)=\big(\sqrt[3]{a^2b} x_1 -\sqrt[3]{c^2a} x_2\big)^2+ \big(\sqrt[3]{a^2b} x_1 -\sqrt[3]{b^2c} x_3\big)^2+\big(\sqrt[3]{c^2a} x_2 -\sqrt[3]{b^2c} x_3\big)^2.
\]

Thus  $\psi(\mathbf{x})\leq 1$ and $\varphi(W_{f,a,b,c}(\mathbf{x}))\leq \varphi(\mathbf{x})$, that is the function $\varphi$ is a Lyapunov function. \\

Now we assume that $\mathbf{x}^{(0)} \in \inte S^2\backslash \{\mathbf{x^*}\}$. Note that $F$ is continuous on $S^2$ and that $F(a,b,c,\mathbf{x})>0$
whenever
$\mathbf{x} \neq \mathbf{x^*}$. Hence $\psi$ is bounded away from 1 outside any neighborhood of $\mathbf{x^*}$, so
$\varphi(\mathbf{x}^{(n)})$ converges to 0 as $n \to \infty$ showing that  $\omega(\xo) \subset  \partial  S^2$. It remains to show that the set
$\omega(\xo)$ is infinite.

Assume that $\omega(\xo)$ is a single point. This point must be a fixed point in $\partial  S^2$, i.e.~one of the points $\mathbf{e}_1,\,\mathbf{e}_2, \mathbf{e}_3$.
If $\omega(\xo)=\{\mathbf{e}_1\}$, then $\xn \to \mathbf{e}_1$, so $x^{(n)}_2\rightarrow 0,x^{(n)}_3\rightarrow 0, \ \ n\rightarrow\infty$.
It follows that $\lim_{n \to \infty}x_2^{(n+1)}/x_2^{(n)}=1-af(\mathbf{e}_1)<1$ and $\lim_{n \to \infty}x_3^{(n+1)}/x_3^{(n)}=1$. Therefore, by  \eqref{cneop},
$x_3^{(n+1)}/x_3^{(n)}>1$ for all sufficiently large $n$ contradicting the fact that $x^{(n)}_3\rightarrow 0$, so $\omega(\xo)$ cannot equal
$\{\mathbf{e}_1\}$. For the same reason,  $\omega(\xo)$ cannot equal $\{\mathbf{e}_2\}$ or $\{\mathbf{e}_3\}$. Further,  $\omega(\xo)$ cannot be equal
to any subset of  $\{\mathbf{e}_1,\mathbf{e}_2, \mathbf{e}_3\}$ of cardinality 2 or 3: assume that this is the case. For any $\varepsilon>0$, the
complement $C$ of the union of the open $\varepsilon$-neighborhoods of  $\mathbf{e}_1$, $\mathbf{e}_2$, and $\mathbf{e}_3$ is compact. If $\varepsilon>0$ is sufficiently small, then no point in the $\varepsilon$-neighborhood of $\mathbf{e}_i$ is mapped to the $\varepsilon$-neighborhood of $\mathbf{e}_j$
whenever $i \neq j$ (by continuity of $W_{f,a,b,c}$). In this case, $C$ contains infinitely many elements of the sequence $\xn$ and hence (by compactness)
at least one element of $\omega(\xo)$ contradicting the assumption.

Consequently,  $\omega(\xo)$ must contain a point in  $\partial  S^2\backslash \{\mathbf{e}_1,\mathbf{e}_2, \mathbf{e}_3\}$, say $\mathbf{y}$. As mentioned before,
the iterates $W^k(\mathbf{y})$, $k \in \N$ converge to one of the three vertices and are all different, so $\omega(\mathbf{x}^0)$ is an infinite set.
The proof of the proposition is complete.
\end{proof}

Consider the following subsets of $S^2$:
\[G_1=\Big\{\mathbf{x}\in S^2: {x_1\over \lambda_1} \geq  {x_2\over \lambda_2}\geq {x_3\over \lambda_3}\Big\}, \, G_2=\Big\{\mathbf{x}\in S^2: {x_1\over \lambda_1} \geq {x_3\over \lambda_3}\geq  {x_2\over \lambda_2}\Big\}, \]
\[G_3=\Big\{\mathbf{x}\in S^2: {x_3\over \lambda_3}\geq {x_1\over \lambda_1} \geq  {x_2\over \lambda_2}\Big\}, \, G_4=\Big\{\mathbf{x}\in S^2: {x_3\over \lambda_3}\geq  {x_2\over \lambda_2}\geq {x_1\over \lambda_1}\Big\},\]
 \[G_5=\Big\{\mathbf{x}\in S^2: {x_2\over \lambda_2}\geq {x_3\over \lambda_3}\geq {x_1\over \lambda_1}\Big\}, \, G_6=\Big\{\mathbf{x}\in S^2: {x_2\over \lambda_2}\geq {x_1\over \lambda_1}\geq {x_3\over \lambda_3}\Big\}. \]
 where $\lambda_1=\sqrt[3]{bc^2}, \, \lambda_2=\sqrt[3]{ab^2}, \, \lambda_3=\sqrt[3]{ca^2}$ as above. For $\gamma \ge 0$, we write $G_i^\gamma:=G_i\bigcap \{\varphi \le \gamma\}$ where $\varphi$ is the Lyapunov function defined in the previous proposition. We write $G_i^\gamma\rightarrow G_j^\gamma$ iff $W_{f,a,b,c}(G_i^\gamma)\subset G_i^\gamma\cup G_j^\gamma$.

 In the following, it will often be convenient to rewrite \eqref{cneop} in terms of the $\lambda_i$ rather than $a,b,c$. Note that
 $$
a=\lambda_2^{1/3} \lambda_3^{4/3}\lambda_1^{-2/3},\;   b=\lambda_1^{1/3} \lambda_2^{4/3} \lambda_3^{-2/3},\; c=\lambda_3^{1/3} \lambda_1^{4/3} \lambda_2^{-2/3}
$$
and therefore
\begin{equation}\label{anders}
  \begin{array}{l}
\frac{x_1'}{\lambda_1}=\frac{x_1}{\lambda_1}\Big(1+\Big(\frac{x_1}{\lambda_1}\frac{x_2}{\lambda_2}-\Big(\frac{x_3}{\lambda_3}\Big)^2\Big)f(\mathbf{x})\lambda_1^{1/3}\lambda_2^{4/3}\lambda_3^{4/3}\Big)\\[2mm]
\frac{x_2'}{\lambda_2}=\frac{x_2}{\lambda_2}\Big(1+\Big(\frac{x_2}{\lambda_2}\frac{x_3}{\lambda_3}-\Big(\frac{x_1}{\lambda_1}\Big)^2\Big)f(\mathbf{x})\lambda_2^{1/3}\lambda_3^{4/3}\lambda_1^{4/3}\Big)\\[2mm]
\frac{x_3'}{\lambda_3}=\frac{x_3}{\lambda_3}\Big(1+\Big(\frac{x_3}{\lambda_3}\frac{x_1}{\lambda_1}-\Big(\frac{x_2}{\lambda_2}\Big)^2\Big)f(\mathbf{x})\lambda_3^{1/3}\lambda_1^{4/3}\lambda_2^{4/3}\Big).
  \end{array}
  \end{equation}

\begin{proposition}\label{per} There exists some $\gamma_0>0$ such that for all $\gamma \in (0,\gamma_0]$ we have
  $G_1^\gamma\to G_2^\gamma\to G_3^\gamma \to G_4^\gamma \to G_5^\gamma \to G_6^\gamma \to G_1^\gamma$. Further, for any
  $\mathbf{x}^{(0)}\in \inte S^2\setminus\{\mathbf{x^*}\}$ and any $i\in\{1,...,6\}$ we have $W_{f,a,b,c}^k(\mathbf{x}^{(0)})\in G_i$ for
  infinitely many  $k \in \N$ and $\{\mathbf{e}_1,\mathbf{e}_2, \mathbf{e}_3\}\subset \omega(\xo)$.
\end{proposition}

\begin{proof} We first show $W_{f,a,b,c}(G_1^\gamma)\subset G_1\cup G_2$ for $\gamma \ge 0$ sufficiently small. Note that $W_{f,a,b,c}(G_1^0)$ is a
  compact subset of the interior $I_{1,2}$ of the set $G_1\cup G_2$ (in the trace topology of $S^2$). By continuity of $W_{f,a,b,c}$ the set
  $W_{f,a,b,c}^{-1}(I_{1,2})$ is an open neighborhood of $G_1^0$. Since $G_1^0=\cap_{\gamma >0}  G_1^\gamma$ and the sets $G_1^\gamma$ are compact, there
  exists some $\gamma_1>0$ such that $W_{f,a,b,c}(G_1^\gamma)\subset G_1 \cup G_2$ for all $\gamma \le \gamma_1$.

  The proof of $W_{f,a,b,c}(G_2^\gamma)\subset G_2\cup G_3$ is similar: $W_{f,a,b,c}(G_2^0)$ is a compact subset of the (relative) interior of $G_2 \cup G_3$.
  The same argument as above shows that there exists some $\gamma_2>0$ such that $W_{f,a,b,c}(G_2^\gamma)\subset G_2 \cup G_3$ for all $\gamma \le \gamma_2$.

  The remaining cases are treated in exactly the same way (resulting in corresponding $\gamma_i>0$). Defining $\gamma_0$ as the smallest of
  the numbers $\gamma_1,...,\gamma_6$,
  we see that $W_{f,a,b,c}(G_i^\gamma)\subset G_i \cup G_{i+1}$ for $\gamma \le \gamma_0$, $i=1,...,6$ (where $G_7:=G_1$). Since $\varphi$ is a Lyapunov function
  we obtain $W_{f,a,b,c}(G_i^\gamma)\subset G_i^\gamma \cup G_{i+1}^\gamma$ for $\gamma \le \gamma_0$,  $i=1,...,6$.

  Now we show that $G_i$ is visited infinitely often. Assume that the trajectory
  starting at $\mathbf{x}\in G_1 \bigcap \inte S^2\setminus\{\mathbf{x^*}\}$ 
  never leaves $G_1$. Then the first coordinate of the trajectory is non-decreasing
  and the second coordinate is non-increasing by \eqref{anders} and hence the trajectory converges to a (fixed) point in $G_1\setminus\{\mathbf{x^*}\}$,
  i.e.~it converges to $\mathbf{e_1}$ contradicting the statement of Proposition \ref{lf}. Next, we assume that the trajectory
  starting at $\mathbf{x}\in G_2 \bigcap \inte S^2\setminus\{\mathbf{x^*}\}$ never leaves $G_2$. Then the third coordinate of the trajectory is non-decreasing
  and the second coordinate is non-increasing by \eqref{anders} and hence the trajectory converges to a (fixed) point in $G_2\setminus\{\mathbf{x^*}\}$,
  but the only fixed point in $G_2\setminus\{\mathbf{x^*}\}$ is $\mathbf{e_1}$ which the sequence cannot converge to since the third
  coordinate is non-decreasing. The corresponding arguments for trajectories starting in $G_i$ for $i \in \{3,4,5,6\}$ are analogous. Therefore, using the
  first part of the proposition and Proposition \ref{lf}, we see that each $G_i$ is visited infinitely many times. Using the previous proposition we see that $\partial S^2 \cap G_i$
  intersects $\omega(\xo)$ and therefore, arguing as in the proof of the previous proposition, $\{\mathbf{e}_1,\mathbf{e}_2, \mathbf{e}_3\}\subset \omega(\xo)$,
  so the proof of the proposition is complete.
\end{proof}

Define the sets $U_1, U_2, U_3$  by
\[U_1=(G_1\cup G_2), \ \ U_2=(G_5\cup G_6), \ \ U_3=(G_3\cup G_4). \]

%
The next proposition provides a lower bound for the sojourn time of the trajectory in $U_i$.

\begin{proposition}\label{staytime}
  Let $i \in \{1,2,3\}$. There exist $A>0$ and $\varepsilon_0>0$ such that
  for every $\varepsilon \in (0,\varepsilon_0]$ there exists some $B_\varepsilon\ge 0$ such that for each $\mathbf{x}^{(0)}\notin U_i$
  such that $\varphi (\mathbf{x}^{(0)})\in (0,\gamma_0]$
  and $n \in \N$ such that $\mathbf{x}^{(s)}\in U_i$ for all $s=1,\dots,n$ and $\mathbf{x}^{(n+1)}\notin U_i$ the number $n(\varepsilon)$ of elements
  in the sequence $\mathbf{x}^{(s)}$,  $s\in \{1,...,n\}$ which are in the $\varepsilon$-neighborhood
  $N_{i,\varepsilon}:=\{\mathbf{x} \in S^2: x_i\ge 1-\varepsilon\}$ of the vertex
  $\mathbf{e}_i$ satisfies
  $$
n(\varepsilon) \ge \frac A\varepsilon \log(1/\varphi(\mathbf{x}^{(0)}))-B_\varepsilon
$$
and there exists $C_\varepsilon$ such that $\inf\big\{ s \in \N:\,\mathbf{x}^{(s)}\in N_{i,\varepsilon}\big\}\le C_\varepsilon$ for all $\mathbf{x}^{(0)}$ as above.
\end{proposition}

\begin{proof} Without loss of generality, we assume that $i=1$. Let $\varepsilon_0>0$ be
  so small that $N_{1,\varepsilon_0} \subset G_1\cup G_2$ and
  \begin{equation}\label{einsgross}
    \Big(\frac {x_1}{\lambda_1}\Big)^2 \ge \frac {x_2}{\lambda_2} \frac {x_3}{\lambda_3}
  \end{equation}
  for all $\mathbf{x} \in N_{1,2\varepsilon_0}$ and let $\varepsilon \in (0,\varepsilon_0]$.
  For ease of notation, we write $x_{k,j}$ instead of $x^{(k)}_j=\big(W_{f,a,b,c}^k(\mathbf{x}^{(0)})\big)_j$.

  Let $m:=m(\varepsilon):=\inf\{s \in \N_0:\,x_{s,3}>\varepsilon/2\}-1$. Note that $m$ is finite by Proposition \ref{per}.  Then, by \eqref{cneop}, for $s\in \{1,.., m\}$,
  $$
\frac{x_{s+1,3}}{x_{s,3}} \le 1 +\varepsilon.
  $$
  Hence $\varepsilon/2 < x_{m+1,3}\le (1+\varepsilon)^m x_{1,3}\le 2 (1+\varepsilon)^m x_{0,3}$ which implies
  $$
m > \frac 1{\log(1+\varepsilon)} \big( \log \varepsilon-\log 4 -\log x_{0,3}\big) \ge  \frac 1\varepsilon \log \frac 1 {x_{0,3}} -C_\varepsilon,
$$
where $C_\varepsilon:=\frac{\log 4 -\log \varepsilon}{\log(1+\varepsilon)}$. Inserting $\varphi(\mathbf{x}^{(0)})$ we get
$$
m >\frac 1\varepsilon \frac 1 \lambda_3 \Big( \log \frac 1{\varphi(\mathbf{x}^{(0)})} + \log \big( x_{0,1}^{\lambda_1} x_{0,2}^{\lambda_2}\big)\Big) -C_\varepsilon.
$$
We need a lower bound for $x_{0,1}$ and $x_{0,2}$. Since $\mathbf{x}^{(0)}\in G_6$ we have
$$
x_{0,2}\ge \frac{\lambda_2}{\lambda_1+\lambda_2+\lambda_3}
$$
and since $\mathbf{x}^{(1)}\in G_1$ we get
$$
x_{0,1}\ge \frac 12 x_{1,1}\ge \frac 12 \frac{\lambda_1}{\lambda_1+\lambda_2+\lambda_3}.
$$
Therefore, there exists some $D_\varepsilon>0$ such that
$$
m \ge  \frac 1\varepsilon \frac 1 \lambda_3 \log \frac 1{\varphi(\mathbf{x}^{(0)})}-D_\varepsilon.
$$
Let $\kappa (\varepsilon):=\inf\{s \in \N: x_{s,2}\le \frac \varepsilon 2\}\wedge m(\varepsilon) -1$ and note that $\mathbf{x}^{(s)} \in N_{1,\varepsilon}$ for all
$s \in \{\kappa(\varepsilon)+1,...,m(\varepsilon)\}$ using \eqref{einsgross} and \eqref{anders}. Therefore,
$$
n(\varepsilon)\ge m(\varepsilon)-\kappa(\varepsilon),
$$
so both claims will follow once we know that for each $\varepsilon \in (0,\varepsilon_0]$,
$\kappa(\varepsilon)$ is bounded with respect to $\mathbf{x}^{(0)}$ as above.

To see that $\kappa (\varepsilon)$ is bounded with respect to $\mathbf{x}^{(0)}$ first note that $x_{s,1}$ is bounded away from 0 uniformly in $\varepsilon \in (0,\varepsilon_0]$ and
$x_{s,2}>\varepsilon/2$ and $x_{s,3}<\varepsilon/2$  for $s \in \{0,...,\kappa(\varepsilon)\}$. Hence, for $\varepsilon_0>0$ sufficiently small and $\varepsilon \in (0,\varepsilon_0]$, there exists some $\Gamma(\varepsilon)>0$ such that $x_{s+1,1}\ge x_{s,1}(1+ \Gamma(\varepsilon))>0$ for all $s \in \{0,...,\kappa(\varepsilon)-1\}$, so
$\kappa(\varepsilon)$ is bounded by a function of $\varepsilon$ and
the proof of the proposition is complete.
\end{proof}

\begin{proof}[Proof of Theorem \ref{mr}] Fix $\mathbf{x} \in \inte S^2\backslash \{\mathbf{x^*}\}$ and $\varepsilon\in(0,\varepsilon_0]$, where $\varepsilon_0>0$ is as
  in the previous proposition. Without loss of generality we (only) show that $\mathbf{e}_1$ is a limit point of the sequence  $\mathbf{c}_k^{(n)}(\mathbf{x})$, $n \in \N_0$ for each $k \in \N_0$
  (for $k=0$ this has already been proved in Proposition \ref{per}).  Let $u_1<v_1<u_2<v_2<...$ be a sequence in $\N$ such that $\varphi(\mathbf{x}^{(u_1)})\le \gamma_0$,
  $\mathbf{x}^{(u_i)} \notin U_1$ and  $\mathbf{x}^{(s)} \in U_1$ for all $s=\{u_i+1,...,v_i\}$ and   $\mathbf{x}^{(v_i+1)} \notin U_1$. Such a sequence exists by Proposition
  \ref{per}. By the previous proposition the number of times up to time $v_i$ the sequence spends in $N_{1,\varepsilon}$ is at least
  $$
n_i(\varepsilon)\ge \frac A{\varepsilon} \log (1/\varphi(\mathbf{x}^{(u_i)}))-B_\varepsilon.
$$
Since $\psi$ (defined in \eqref{psi}) is bounded away from 1 outside a neighborhood of $\mathbf{x}^*$ there exists some $\delta>0$ (depending on the strarting point $\mathbf{x}$)
such that $\varphi(\mathbf{x}^{(l)})\le (1-\delta)^l\varphi(\mathbf{x})$ for all $l \in \N$. Therefore,
$$
n_i(\varepsilon)\ge \frac A{\varepsilon}\Big( u_i\log \frac 1{1-\delta} + \log (1/\varphi(\mathbf{x}))\Big)   -B_\varepsilon.
$$
Using the last statement in the previous proposition, we see that the proportion of time which the orbit starting from $\mathbf{x}$ spends in $N_{1,\varepsilon}$  is asymptotically
(as $i \to \infty$) bounded from below by
$$
\Big(1+\frac 1{\frac A \varepsilon \log \frac 1{1-\delta}}\Big)^{-1}.
$$
Choosing $\varepsilon>0$ sufficiently small the assertion of the theorem follows from \eqref{Cesaroproperty}.
\end{proof}


Finally, we investigate the case in which all parameters $a,b,c$ are negative. In this case we get completely different asymptotic properties.

\begin{theorem}\label{mr2}
  If $a<0$, $b<0$, and $c<0$, then
  $$
  \lim_{k \to \infty} W^k_{f,a,b,c}(\mathbf{x})={\mathbf{x^*}}
  $$
  whenever $\mathbf{x} \in \inte S^2$,
  where
  $$
\mathbf{x^*}:=\bigg({\lambda_1\over \lambda_1+\lambda_2+\lambda_3},{\lambda_2\over \lambda_1+\lambda_2+\lambda_3},{\lambda_3\over \lambda_1+\lambda_2+\lambda_3}\bigg)
$$
and $\lambda_1=\sqrt[3]{|bc^2|}, \, \lambda_2=\sqrt[3]{|ab^2|}, \, \lambda_3=\sqrt[3]{|a^2c|}$.
\end{theorem}

\begin{proof}
	We restrict our proof to the case $f(\mathbf{x}) \leq \frac{5}{4} \min
	\left\{ \frac{\lambda_i}{\lambda_j} : i,j \in \left\{1,2,3\right\} \right\}$.
	For the general case see \cite{genPar}.
	Using the same Lyapunov function as in Proposition \ref{lf}, it is enough to show that
	\begin{align*}
		\psi(\mathbf{x})=
		\left( 1 + \left( ax_1x_2 - bx_3^2 \right) f (\mathbf{x}) \right)^{\lambda_1}
		\left( 1 + \left( cx_2x_3 - ax_1^2 \right) f (\mathbf{x}) \right)^{\lambda_2}
		\left( 1 + \left( bx_3x_1 - cx_2^2 \right) f (\mathbf{x}) \right)^{\lambda_3}
		>1
	\end{align*}
	whenever $\mathbf{x} \in \inte S^2\setminus \{{\mathbf{x^*}}\}$.
	For simplicity we set
	\begin{align*}
		u= \sqrt[3]{|a^2b|} x_1 \sqrt{f(\mathbf{x})} \qquad
		v= \sqrt[3]{|ac^2|} x_2 \sqrt{f(\mathbf{x})} \qquad
		w= \sqrt[3]{|b^2c|} x_3 \sqrt{f(\mathbf{x})}.
	\end{align*}
	Without loss of generality we let
	$\lambda_1 = \min \left\{ \lambda_1, \lambda_2, \lambda_3 \right\}$.
	Since $\lambda \mapsto (1 + \lambda^{-1} \xi )^\lambda$ is non-decreasing
	on $[\lambda_1,1]$ for all $\xi \in (-\lambda_1, \infty)$, it follows that
	\begin{align*}
		\psi (\mathbf{x}) &=
		\left( 1 + \lambda_1^{-1} \left( w^2-uv \right) \right)^{\lambda_1}
		\left( 1 + \lambda_2^{-1} \left( u^2-vw \right) \right)^{\lambda_2}
		\left( 1 + \lambda_3^{-1} \left( v^2-uw \right) \right)^{\lambda_3}
		\\& \geq
		\left( 1 + \lambda_1^{-1} \left( w^2-uv \right) \right)^{\lambda_1}
		\left( 1 + \lambda_1^{-1} \left( u^2-vw \right) \right)^{\lambda_1}
		\left( 1 + \lambda_1^{-1} \left( v^2-uw \right) \right)^{\lambda_1}
		.
	\end{align*}
	Using monotonicity and concavity of the logarithmic function, we obtain
	\begin{align*}
		\frac{1}{3 \lambda_1} \ln \psi (\mathbf{x}) &\geq - \left(
			 \frac{1}{3} \ln \frac{1}{1 + \lambda_1^{-1} \left( w^2-uv \right)}
			+\frac{1}{3} \ln \frac{1}{1 + \lambda_1^{-1} \left( u^2-vw \right)}
			+\frac{1}{3} \ln \frac{1}{1 + \lambda_1^{-1} \left( v^2-uw \right)} \right) \\
		&\geq - \ln \left( \frac{1}{3} \left(
			 \frac{1}{1 + \lambda_1^{-1} \left( w^2-uv \right)}
			+\frac{1}{1 + \lambda_1^{-1} \left( u^2-vw \right)}
			+\frac{1}{1 + \lambda_1^{-1} \left( v^2-uw \right)} \right) \right) \\
		&\geq - \ln \left( 1 -
			\frac{F(\mathbf{x})}{3
			\left(1 + \lambda_1^{-1} \left( w^2-uv \right) \right)
			\left(1 + \lambda_1^{-1} \left( u^2-vw \right) \right)
			\left(1 + \lambda_1^{-1} \left( v^2-uw \right) \right) } \right)
	\end{align*}
	where
	\begin{align*}
		F( \mathbf{x} ) &= \lambda_1^{-1} \left( u^2 +v^2 +w^2 - uv -uw -vw \right) \\
			&\quad + 2\lambda_1^{-2} \left( (w^2 -uv)(u^2-vw) + (w^2-uv) (v^2 -uw)
				+ (u^2 -vw) (v^2 -uw) \right) \\
			& \quad + \lambda_1^{-3} (w^2 -uv) (u^2 -vw) (v^2-uw).
	\end{align*}
	It remains to show that $ F( \mathbf{x} ) >0$ for all
	$\mathbf{x} \in \inte S^2\setminus \{{\mathbf{x^*}}\}$.
	Therefore, we write
	$F( \mathbf{x} ) = \lambda_1^{-1} F_1( \mathbf{x} )+
	2 \lambda_1^{-2} F_2( \mathbf{x} ) +\lambda_1^{-3} F_3( \mathbf{x} )$
	and estimate the terms $F_1( \mathbf{x} ), F_2( \mathbf{x} ), F_3( \mathbf{x} )$
	separately. For any $\mathbf{x} \in \inte S^2\setminus \{{\mathbf{x^*}}\}$,
	the first term is positive since
	\begin{align*}
		F_1( \mathbf{x} ) =  u^2 +v^2 +w^2 - uv -uw -vw
		= \frac{1}{2} (u-v)^2 + \frac{1}{2} (u-w)^2 + \frac{1}{2} (v-w)^2 >0.
	\end{align*}
	The second term can be estimated by
	\begin{align*}
		F_2( \mathbf{x} ) &=(w^2 -uv)(u^2-vw) +(w^2-uv) (v^2 -uw) +(u^2 -vw) (v^2 -uw) \\
		&= - F_1( \mathbf{x} ) (uv +uw +vw) \\
		&\geq - F_1( \mathbf{x} )
			( \lambda_1 x_1 x_2 f(\mathbf{x}) +\lambda_3 x_1x_3 f(\mathbf{x})
			+ \lambda_2 x_2x_3 f(\mathbf{x}) ) \\
		&\geq - \frac{5}{4} \lambda_1 F_1 ( \mathbf{x} )  (x_1x_2 +x_1x_3+x_2x_3)
		\geq - \frac{5}{12} \lambda_1 F_1( \mathbf{x} ).
	\end{align*}
	Estimating the third and last term, we obtain
	\begin{align*}
		F_3( \mathbf{x} ) &= (w^2 -uv) (u^2 -vw) (v^2-uw) \\
		&=			  \frac{1}{2} (u-v)^2 (-uw^3 +u^2vw -vw^3 +uv^2w)\\
			& \quad + \frac{1}{2} (u-w)^2 (-uv^3 +u^2vw -v^3w +uvw^2)\\
			& \quad + \frac{1}{2} (v-w)^2 (-u^3v +uv^2w -u^3w +uvw^2)\\
		&\geq -\frac{1}{2} (u-v)^2 (\lambda_1\lambda_3 x_1x_3^3 f(\mathbf{x})^2
				+ \lambda_1\lambda_2x_2x_3^3 f(\mathbf{x})^2)\\
			& \quad - \frac{1}{2} (u-w)^2 (\lambda_1\lambda_3x_1x_2^3  f(\mathbf{x})^2
				+ \lambda_2\lambda_3 x_2^3x_3 f(\mathbf{x})^2)\\
			& \quad - \frac{1}{2} (v-w)^2 (\lambda_1\lambda_2x_1^3x_2  f(\mathbf{x})^2
				+\lambda_2\lambda_3x_1^3x_3 f(\mathbf{x})^2)\\
		&\geq -  \frac{25}{16} \lambda_1^2  \Big(
			  \frac{1}{2} (u-v)^2 (x_1x_3^3 +x_2x_3^3))
			+ \frac{1}{2} (u-w)^2  (x_1x_2^3 +x_2^3x_3) \\
			& \quad + \frac{1}{2} (v-w)^2  (x_1^3x_2 +x_1^3x_3) \Big)\\
		&\geq -\frac{25}{16} \frac{27}{256} \lambda_1^2 F_1( \mathbf{x})
		> -\frac{1}{6} \lambda_1^2 F_1( \mathbf{x}).
	\end{align*}
	for any $\mathbf{x} \in \inte S^2\setminus \{{\mathbf{x^*}}\}$. To conclude
	\begin{align*}
		F( \mathbf{x} ) >
		\lambda_1^{-1} F_1( \mathbf{x} ) \left(1- \frac{5}{6} - \frac{1}{6} \right) =0
	\end{align*}
	for any $\mathbf{x} \in \inte S^2\setminus \{{\mathbf{x^*}}\}$, so the proof is complete.

%
\end{proof}

\section{Remarks concerning the continuous-time case}\label{conticase}

If we replace $f(x)$ by $f(x)/n$ in \eqref{cneop}, $n \in \N$,  then the resulting recursion is the Euler scheme for the ode
\begin{align*}
  \frac {dx_1}{dt}&=x_1(ax_1x_2-bx_3^2)f(\mathbf{x})\\ 
  \frac {dx_2}{dt}&=x_2(cx_2x_3-ax_1^2)f(\mathbf{x})\\
  \frac {dx_3}{dt}&=x_3(bx_1x_3-cx_2^2)f(\mathbf{x}).\\
\end{align*}
Therefore, as $n \to \infty$, the corresponding solutions converge uniformly on compact time intervals (in the usual sense).
Obviously, the fixed points remain
unchanged as $n$ changes. We will not provide detailed arguments showing that the asymptotics also remain unchanged in the limit. We just mention that
the function $\varphi(\mathbf{x})=x_1^{\lambda_1} x_2^{\lambda_2} x_3^{\lambda_3}$ defined in Proposition \ref{lf} remains a Lyapunov function when $n \to \infty$
(i.e.~$\langle \varphi(\mathbf{x}),\frac{dx}{dt}\rangle<0$ for $\mathbf{x} \in \inte(S^2)\backslash\{\mathbf{x}^*\}$) 
and the proof is even easier than in the discrete case.

\section{Conclusion}\label{sec:conclusion}

We considered  a prey-predator model with three interacting species defined by \eqref{cneop} and studied its
asymptotic behavior. If $a,b$ and $c$ are all positive, then  we have  $\min\limits_{1\leq i\leq3} x_i^{(n)}\rightarrow 0$, when $n\rightarrow \infty $ for every starting point
in the interior of the simplex except for the unique fixed point (Proposition \ref{lf}), while Proposition \ref{staytime} states that $ \lim\sup_{t\rightarrow\infty} x_i(t)>0$
for any  initial $\mathbf{x}\in \inte S^2$ and for all $i=1,2,3$, that is the biological system is weakly persistent (in the sense of
Definition \ref{pgh}) but not strongly persistent.
In this case the biological system describes a kind of circular prey-predator model and Theorem  \ref{mr}
shows non-ergodic behaviour in which all species survive but for most of the time one of the three species dominates while the other two species are almost extinct.

If all parameters are negative, then the population stabilizes towards an equilibrium in which all three species survive. 
In this case the biological system is strongly persistent (Theorem \ref{mr2}) in sense of Definition \ref{pfw}.

If two of the parameters have a different sign then, in the limit, only one species survives, that is
the biological system is not even weakly persistent.

%
%

\section*{Acknowledgments}
This work was done in the Technische Universit\"at (TU) Berlin, Germany. The first author (UJ)  thanks the TU Berlin,  for the kind hospitality and for  providing all facilities and the  German Academic Exchange Service (DAAD) for providing financial support by a scholarship.

\end{document}